\documentclass{daj}

\dajAUTHORdetails{%
  title = {Product Mixing in the Alternating Group}, 
  author = {Sean Eberhard},
  plaintextauthor = {Sean Eberhard},
  keywords = {symmetric group, alternating group, product-free sets, mixing in groups},
}   

\dajEDITORdetails{%
   year={2016},
   number={2},
   received={14 December 2015},   
   published={28 February 2016},  
   doi={10.19086/da.610},       
}   

\usepackage{amscd}
\usepackage{amsfonts}
\usepackage{amsmath}
\usepackage{amssymb}
\usepackage{amsthm}
\usepackage{setspace}
\usepackage{version}
\usepackage{url}

\newtheorem{theorem}{Theorem}[section]
\newtheorem{lemma}[theorem]{Lemma}
\newtheorem{proposition}[theorem]{Proposition}

\renewcommand{\(}{\left(}
\renewcommand{\)}{\right)}

\def\tint{{\textstyle\int}}

\def\C{\mathbf{C}}
\def\E{\mathbf{E}}
\def\P{\mathbf{P}}

\DeclareMathOperator{\SL}{SL}
\DeclareMathOperator{\HS}{HS}
\DeclareMathOperator{\tr}{tr}

\def\HS{\textup{HS}}

\numberwithin{equation}{section}

\begin{document}

\begin{frontmatter}[classification=text]


\author[eberhard]{Sean Eberhard}

\begin{abstract}
We prove the following one-sided product-mixing theorem for the alternating group: Given subsets $X,Y,Z \subset A_n$ of densities $\alpha,\beta,\gamma$ satisfying
\[
  \min(\alpha\beta,\alpha\gamma,\beta\gamma)\gg n^{-1}(\log n)^7,
\]
there are at least
\[
  (1+o(1))\alpha\beta\gamma |A_n|^2
\]
solutions to $xy=z$ with $x\in X, y\in Y, z\in Z$. One consequence is that the largest product-free subset of $A_n$ has density at most $n^{-1/2}(\log n)^{7/2}$, which is best possible up to logarithms and improves the best previous bound of $n^{-1/3}$ due to Gowers. The main tools are a Fourier-analytic reduction noted by Ellis and Green to a problem just about the standard representation, a Brascamp--Lieb-type inequality for the symmetric group due to Carlen, Lieb, and Loss, and a concentration of measure result for rearrangements of inner products.
\end{abstract}
\end{frontmatter}








\section{Introduction}

Product mixing for a group $G$ generally refers to any estimate of the following form: whenever subsets $X,Y,Z\subset G$ have densities $\alpha,\beta,\gamma$ above some threshold, the number of solutions to $xy=z$ with $x\in X,y\in Y,z\in Z$ is $(1+o(1))\alpha\beta\gamma|G|^2$. The following foundational theorem proved by Gowers~\cite{gowers} (and expanded by Babai, Nikolov, and Pyber~\cite{babainikolovpyber}) explains this idea further.

\begin{theorem}[Gowers]\label{gowers}
Let $G$ be a group and let $m$ be the minimal dimension of a nontrivial representation of $G$. Let $X,Y,Z\subset G$ have densities $\alpha,\beta,\gamma$, respectively. Then
\[
 \left| \langle 1_X*1_Y,1_Z\rangle - \alpha\beta\gamma \right| < m^{-1/2} \alpha^{1/2} \beta^{1/2} \gamma^{1/2}.
\]
In particular if $\alpha\beta\gamma \gg m^{-1}$ then
\[
  \langle 1_X*1_Y,1_Z\rangle = (1+o(1))\alpha\beta\gamma.
\]
\end{theorem}

Here and throughout this paper we write $X\lesssim Y$ to mean that $X\leq O(Y)$, and we write $X\ll Y$ to mean that $X\leq o(Y)$. We will write $X\sim Y$ to mean $X\lesssim Y$ and $X\gtrsim Y$. This differs from the standard convention in analytic number theory, but it will be convenient for us.

There are several immediate corollaries of Theorem~\ref{gowers}. For example, if $\alpha\beta\gamma\geq m^{-1}$, then the intersection $XY\cap Z$ is nonempty, and in fact $XYZ^{-1} = G$. In particular, if $X\subset G$ is product-free (meaning that there are no solutions to $xy=z$ with $x,y,z\in X$), then $X$ has density at most $m^{-1/3}$.

For the purpose of illustration let us assume $\alpha\sim\beta\sim\gamma$. Then Theorem~\ref{gowers} asserts that there is a product-mixing phenomenon for sets of density greater than $m^{-1/3}$. On the other hand Kedlaya~\cite{kedlaya1} proved that any group $G$ acting transitively on a set of size $n$ has a product-free subset of density $n^{-1/2}$. For a broad class of groups, including for example the alternating groups and special linear groups, we have $m\sim n$, so for these groups this leaves a gap between $m^{-1/3}$ and $m^{-1/2}$.

In Section~\ref{examples} we partly explain this gap by showing that any group $G$ acting transitively on a set of size $n$ has a subset $X$ of density $\sim n^{-1/3}$ for which there are significantly \emph{more} than the expected number of solutions to $xy=z$. In groups with $m\sim n$ this shows that the density threshold for product mixing is $m^{-1/3}$, as in Gowers's theorem. 

Our main purpose, however, is to demonstrate that a one-sided product-mixing phenomenon persists in the alternating group $A_n$ for somewhat lower densities. Specifically we prove the following theorem.

\begin{theorem}\label{main}
If $X,Y,Z\subset A_n$ have densities $\alpha,\beta,\gamma$, respectively, and
\[
  \min(\alpha\beta,\alpha\gamma,\beta\gamma) \gg \frac{(\log n)^7}{n},
\]
then
\[
 \langle 1_X*1_Y,1_Z\rangle \geq (1+o(1)) \alpha\beta\gamma.
\]
\end{theorem}

As a corollary we deduce that if $X$ has density $\gg n^{-1/2}(\log n)^{7/2}$ then $X^2$ has density
\[
  1 - O(n^{-1/2}(\log n)^{7/2}).
\]
In particular if $X$ is product-free then $X$ has density at most $O(n^{-1/2}(\log n)^{7/2})$. This is best possible up to the logarithmic factors.

As to the methods, we first use nonabelian Fourier analysis to reduce to a problem taking place only in the standard representation, an idea due to Ellis and Green. This problem is then interpretted in terms of random rearrangements of inner products, and we tackle this problem using concentration of measure and entropy subadditivity. The backbone of our proof is a Brascamp--Lieb-type inequality for the symmetric group due to Carlen, Lieb, and Loss, which we explain in Section~\ref{sec:cll}.

\emph{Notation.} As already mentioned, in addition to the usual asymptotic notation $O(\cdot)$ and $o(\cdot)$, we write $X\lesssim Y$ to mean that $X\leq O(Y)$, and we write $X\ll Y$ to mean that $X\leq o(Y)$. We write $X\sim Y$ to mean $X\lesssim Y$ and $X\gtrsim Y$.

We write $\Omega$ throughout for the ground set $\{1,\dots,n\}$ on which $S_n$ and $A_n$ act. We attach the uniform measures to $S_n$, $A_n$, and $\Omega$, and we write an unadorned integral $\tint f$ to mean the integral with respect to the uniform measure on the domain of $f$. We also define inner products, $L^p$ norms, and convolutions accordingly.

\emph{Acknowledgments.} I thank David Ellis for pointing out several confusing typos and recommending a few improvements. 

\section{Examples of sets with poor product mixing}\label{examples}

In this section we give two concrete examples of fairly dense sets with poor product-mixing properties. The first example, a relatively large product-free set, is due to Kedlaya~\cite{kedlaya1} and independently Edward Crane (Ben Green, personal communication), but we recall the construction here as it shows that Theorem~\ref{main} is best possible up to logarithms. The second construction is original, and shows that Theorem~\ref{gowers} is best possible for two-sided mixing.

\subsection{Sets with no solutions to $xy=z$}

First we give an example of a set $X$ of density $\sim n^{-1/2}$ with no solutions to $xy=z$. Fix a set $T\subset\Omega$ of size $t$ and a point $1$ not in $T$ and let $X$ be the set of all $\pi\in A_n$ such that $\pi(1)\in T$ and such that $\pi(T)\subset T^c$. Then clearly $X^2$ is disjoint from $X$, as every $\pi\in X^2$ satisfies $\pi(1)\in T^c$, and it is straightforward to see that $X$ has density
\[
  \frac1{n!} t \binom{n-t}{t} t! (n-t-1)! = \frac{t (n-t)! (n-t-1)!}{n! (n-2t)!} = \frac{t}{n} e^{O(t^2/n)}.
\]
Thus if $t\sim n^{1/2}$ then $X$ has density $\sim n^{-1/2}$. This example is due to Kedlaya~\cite{kedlaya1} and independently Edward Crane (Ben Green, personal communication), and it shows that Theorem~\ref{main} is best possible up to logarithms.

As explained in Kedlaya~\cite{kedlaya1}, his construction adapts straightforwardly to any $2$-transitive subgroup $G\leq S_n$, and in fact it adapts to any transitive subgroup $G\leq S_n$ through an averaging argument.

\begin{proposition}[Kedlaya~\cite{kedlaya1}]\label{prop:ked}
Let $G$ be a transitive subgroup of $S_n$. Then there is a subset $X\subset G$ of density $\sim n^{-1/2}$ such that $X^2 \cap X = \emptyset$.
\end{proposition}
%

\subsection{Sets with too many solutions to $xy=z$}

Next we give an example of a set $X$ of density $\alpha\sim n^{-1/3}$ having many more than the expected number of solutions (namely, $\alpha^3 n!^2$) to $xy=z$. Fix a set $T$ of size $t$ and let $X$ be the set of all $\pi\in A_n$ such that $\pi(T) \cap T$ is nonempty. As long as $t = o(n^{1/2})$ then $X$ has density roughly $t^2/n$, and if you choose $\pi_1,\pi_2$ randomly from $X$ then $\pi_1\pi_2$ is again in $X$ with probability of order $t^2/n + 1/t$. To see this it may help to notice that $X$ is symmetric, and that $\pi_1^{-1} \pi_2 \in X$ if and only if $\pi_1(T) \cap \pi_2(T) \neq \emptyset$. Each of $\pi_1(T)$ and $\pi_2(T)$ is required to intersect $T$ nontrivially, so $\pi_1(T)$ and $\pi_2(T)$ intersect with probability at least $1/t$. Aside from that restriction $\pi_1(T)$ and $\pi_2(T)$ are just random sets of size $t$, so they intersect with probability at least $t^2/n$. (We can afford to be somewhat lax with this computation as we will shortly prove a more general proposition.) Note that the probability $t^2/n + 1/t$ is much larger than the expected probability $t^2/n$ whenever $t$ is small compared to $n^{1/3}$. 

As with the previous construction, this construction adapts straightforwardly to any $2$-transitive subgroup $G\leq S_n$, and to an arbitrary transitive subgroup $G\leq S_n$ through an averaging argument.

\begin{proposition}\label{prop:many}
Let $G$ be a transitive subgroup of $S_n$. Then there is a subset $X\subset G$ of density $\alpha \sim n^{-1/3}$ for which there are at least $100\alpha^3|G|^2$ solutions to $xy=z$ with $x,y,z\in X$.
\end{proposition}
\begin{proof}
For $T\subset \Omega$ of size $t$ let $X_T$ be the set of all $g\in G$ for which $g(T)\cap T\neq\emptyset$. Also fix an arbitrary total order $<$ on $\Omega$. Clearly $|X_T|/|G| \leq t^2/n$. We will bound $|X_T|$ below by the number of $g\in G$ for which there are $i,j\in T$ with $i<j$ such that $g(i)=j$. Thus by inclusion-exclusion we have
\begin{align*}
  \frac{|X_T|}{|G|}
  &\geq \sum_{\substack{i,j\in T\\i<j}} \frac{|\{g:g(i) = j\}|}{|G|} - \sum_{\substack{i,j,i',j'\in T \\ i<j,i'<j' \\ (i,j) \neq (i',j')}} \frac{|\{g:g(i)=j,g(i')=j'\}|}{|G|}.
\end{align*}
The first sum here is $\sim t^2/n$ by transitivity, for any $T$. The second sum can be rewritten as
\begin{equation}\label{secondsum}
  \sum_{\substack{i,i'\in\Omega \\ i \neq i'}} \frac1{|G|} \sum_{\substack{g\in G\\g(i)>i,g(i')>i'}} 1_{i\in T} 1_{g(i)\in T} 1_{i'\in T} 1_{g(i')\in T}.
\end{equation}
Now note that for any fixed $i,i'\in\Omega$ such that $i\neq i'$ and for any $g$ satisfying $g(i)>i$ and $g(i')>i'$ we have $|\{i,g(i),i',g(i')\}|\geq 3$, and in fact $|\{i,g(i),i',g(i')\}|=4$ except for a proportion at most $O(1/n)$ of $g\in G$. It follows that the average of~\eqref{secondsum} over $T\subset\Omega$ is bounded by
\[
  O(n^2(t/n)^4 + n(t/n)^3) = O(t^4/n^2).
\]
Thus, by Markov's inequality, \eqref{secondsum} is $O(t^4/n^2)$ with probability at least $9/10$.

Similarly let us count solutions to $xy=z$ in $X_T$. We will bound the number $N_T$ of solutions below by the number of pairs $(g_1,g_2)\in G^2$ for which there exists $i,j,k\in T$ with $i<j<k$ such that $g_1(i)=j$ and $g_2(j)=k$. Thus by inclusion-exclusion again we have
\begin{align*}
  \frac{N_T}{|G|^2}
  &\geq \sum_{\substack{i,j,k\in T\\ i<j<k}} \frac{|\{(g_1,g_2)\in G^2: g_1(i)=j,g_2(j)=k\}|}{|G|^2}\\
  &\, - \sum_{\substack{i,j,k,i',j',k'\in T\\ i<j<k,i'<j'<k'\\ (i,j,k)\neq(i',j',k')}} \frac{|\{(g_1,g_2)\in G^2: g_1(i)=j, g_2(j)=k, g_1(i')=j', g_2(j')=k'\}|}{|G|^2}.
\end{align*}
The first sum is $\sim t^3/n^2$ by transitivity. The second sum can be rewritten
\begin{equation}\label{secondsum2}
  \sum_{\substack{j,j'\in\Omega\\j\neq j'}} \frac1{|G|^2} \sum_{\substack{g_1,g_2\in G\\g_1^{-1}(j)<j<g_2(j)\\g_1^{-1}(j')<j'<g_2(j')}} 1_{g_1^{-1}(j)\in T} 1_{j\in T} 1_{g_2(j)\in T} 1_{g_1^{-1}(j')\in T} 1_{j'\in T} 1_{g_2(j')\in T}.
\end{equation}
To bound this we again average over $T\subset\Omega$. For $j\neq j'$ and $g_1,g_2$ under the stated restrictions the set
\[
  S = \{g_1^{-1}(j),j,g_2(j),g_1^{-1}(j'),j',g_2(j')\}
\]
always has size at least $4$, has size $4$ for at most a proportion $O(1/n^2)$ of $(g_1,g_2)\in G^2$, has size $5$ for at most a proportion $O(1/n)$ of $(g_1,g_2)\in G^2$, and otherwise has size $6$. It follows that the average of~\eqref{secondsum2} over $T\subset\Omega$ is bounded by
\[
  O(n^2(t/n)^6 + n(t/n)^5 + (t/n)^4) = O(t^6/n^4).
\]
Thus, by Markov's inequality, \eqref{secondsum2} is $O(t^6/n^4)$ with probability at least $9/10$.

We deduce that there is some $T$ for which \eqref{secondsum} is $O(t^4/n^2)$ and \eqref{secondsum2} is $O(t^6/n^4)$. For this $T$ it follows that
\[
  \frac{|X_T|}{|G|} \sim t^2/n + O(t^4/n^2)
\]
and that
\[
  \frac{N_T}{|G|^2} \gtrsim t^3/n^2 + O(t^6/n^4).
\]
Thus as long as $t=o(n^{1/2})$ we see that $X_T$ has density $\alpha \sim t^2/n$ while there are at least $(t^3/n^2)|G|^2 \sim (n/t^3) \alpha^3 |G|^2$ solutions to $xy=z$ in $X$. Now take $t=\lfloor c n^{1/3}\rfloor$ for a sufficiently small constant $c$.
\end{proof}

\section{Nonabelian Fourier analysis}

Here we briefly recall the fundamentals of nonabelian Fourier analysis, and then we give a short Fourier-analytic proof of Theorem~\ref{gowers}. This proof seems to be well known among experts: see for example Wigderson~\cite[Chapter 2.11]{wigderson}.

Let $G$ be a compact group endowed with the uniform measure. The Fourier transform of a function $f\in L^2(G)$ at an irreducible unitary representation $\xi:G\to U(d_\xi)$ is defined by
\[
  \hat{f}(\xi) = \int_G f(x) \xi(x).
\]
We then have the inversion formula
\[
  f(x) = \sum_\xi d_\xi \langle\hat{f}(\xi), \xi(x)\rangle_\textup{HS},
\]
and Parseval's identity
\begin{equation}\label{parseval}
  \langle f,g\rangle = \sum_\xi d_\xi \langle\hat{f}(\xi),\hat{g}(\xi)\rangle_\textup{HS}.
\end{equation}
Here the sums are taken over a complete set of representatives of the irreducible representations of $G$ up to equivalency, and the Hilbert--Schmidt inner product $\langle\cdot,\cdot\rangle_\textup{HS}$ is defined by
\[
  \langle R,S\rangle_\textup{HS} = \tr(RS^*).
\]
Like classical Fourier analysis, nonabelian Fourier analysis is a powerful tool for understanding the behaviour of convolutions. Here the convolution $f*g$ of two functions $f,g\in L^2(G)$ is defined by
\begin{equation}\label{convolution-definition}
  f*g(x) = \int_G f(y) g(y^{-1}x),
\end{equation}
and by an application of Fubini's theorem we have the rule
\begin{equation}\label{convolutionrule}
  \widehat{f*g}(\xi) = \hat{f}(\xi)\hat{g}(\xi).
\end{equation}
For all this and more the reader might refer to Tao~\cite[\textsection 2.8]{tao}.

We can now give a short proof of Theorem~\ref{gowers}.

\begin{proof}[Proof of Theorem~\ref{gowers}]
Suppose that $G$ is finite, that $d_\xi\geq m$ for $\xi\neq 1$, and that $X,Y,Z\subset G$ have densities $\alpha,\beta,\gamma$, respectively. Let $f=1_X, g=1_Y, h=1_Z$. Then by the convolution rule~\eqref{convolutionrule} and Parseval~\eqref{parseval} we have
\begin{align*}
  \langle f*g,h\rangle
  &= \sum_\xi d_\xi \langle \hat{f}(\xi)\hat{g}(\xi),\hat{h}(\xi)\rangle_\textup{HS}\\
  &= \alpha\beta\gamma + \sum_{\xi\neq 1} d_\xi \langle\hat{f}(\xi)\hat{g}(\xi),\hat{h}(\xi)\rangle_\textup{HS}.
\end{align*}
Here we have written $1$ for the trivial representation of $G$. Now by Cauchy--Schwarz and the algebra property $\|RS\|_\textup{HS}\leq\|R\|_\textup{HS}\|S\|_\textup{HS}$ of the Hilbert--Schmidt norm we have
\[
  |\langle\hat{f}(\xi)\hat{g}(\xi),\hat{h}(\xi)\rangle_\HS| \leq \|\hat{f}(\xi)\hat{g}(\xi)\|_\textup{HS} \|\hat{h}(\xi)\|_\textup{HS} \leq \|\hat{f}(\xi)\|_\textup{HS} \|\hat{g}(\xi)\|_\textup{HS} \|\hat{h}(\xi)\|_\textup{HS},
\]
so by using Cauchy--Schwarz together with Parseval again we have
\begin{equation}\label{mixcomp}
\begin{aligned}
  \sum_{\xi\neq 1} d_\xi |\langle\hat{f}(\xi)\hat{g}(\xi),\hat{h}(\xi)\rangle_\textup{HS}|
  &\leq \sum_{\xi\neq 1} d_\xi \|\hat{f}(\xi)\|_\textup{HS} \|\hat{g}(\xi)\|_\textup{HS} \|\hat{h}(\xi)\|_\textup{HS}\\
  &\leq \max_{\xi\neq 1} \|\hat{f}(\xi)\|_\textup{HS} \sum_\xi d_\xi \|\hat{g}(\xi)\|_\textup{HS} \|\hat{h}(\xi)\|_\textup{HS}\\
  &\leq m^{-1/2} \|f\|_2 \|g\|_2 \|h\|_2\\
  &=m^{-1/2} \alpha^{1/2} \beta^{1/2} \gamma^{1/2}.
\end{aligned}
\end{equation}
This proves Theorem~\ref{gowers}.
\end{proof}

For the rest of the paper we specialize to the alternating group $G = A_n$. As explained in the introduction, Theorem~\ref{gowers} provides a satisfactory estimate for $\langle 1_X*1_Y,1_Z\rangle$ only if $\alpha\beta\gamma\gg1/n$. However, as observed by Ellis and Green (personal communication), by examination of the proof above it is clear that only the standard $(n-1)$-dimensional representation $\sigma$ is problematic: Again taking $f=1_X, g=1_Y, h = 1_Z$, we have
\begin{align}
	\langle f*g,h\rangle
	&= \sum_{\xi} d_\xi \langle \hat{f}(\xi)\hat{g}(\xi),\hat{h}(\xi)\rangle_\HS\nonumber\\
	&= \alpha\beta\gamma + (n-1)\langle\hat{f}(\sigma)\hat{g}(\sigma),\hat{h}(\sigma)\rangle_\HS + \sum_{\xi\neq 1,\sigma} d_\xi\, \langle\hat{f}(\xi)\hat{g}(\xi),\hat{h}(\xi)\rangle_\HS, \label{ellisgreen}
\end{align}
and since $d_\xi\gtrsim n^2$ for $\xi\neq 1,\sigma$ (this follows from the hook formula: see for example~\cite[Result~2]{rasala}) we have, by straightforward adaptation of~\eqref{mixcomp},
\[
  \sum_{\xi\neq 1,\sigma} d_\xi |\langle\hat{f}(\xi)\hat{g}(\xi),\hat{h}(\xi)\rangle_\HS| \lesssim n^{-1} \alpha^{1/2}\beta^{1/2}\gamma^{1/2}.
\]
This is negligible compared to the main term $\alpha\beta\gamma$ whenever $\alpha\beta\gamma \gg n^{-2}$. Thus it remains only to control $(n-1)\langle \hat{f}(\sigma)\hat{g}(\sigma),\hat{h}(\sigma)\rangle_\HS$.

For each $i\in\Omega$ we have a map $S_n\to\Omega$ given by $\pi\mapsto \pi(i)$, which induces a map $L^2(\Omega)\to L^2(S_n)$ given by composition with $\pi\mapsto \pi(i)$. We denote by $p_i$ the adjoint of this map, and we call $p_i f$ the \emph{pushforward} of $f$ under $\pi\mapsto\pi(i)$. Explicitly $p_if$ is defined by
\[
  p_i f(\omega) = n \int_{S_n} f(\pi) 1_{\pi(i)=\omega} = \frac1{(n-1)!} \sum_{\substack{\pi\in S_n\\\pi(i)=\omega}} f(\pi),
\]
and for any $g\in L^2(\Omega)$ we have
\[
  \int_{S_n} f(\pi) g(\pi(i)) = \int_\Omega p_if(\omega) g(\omega).
\]

Now by direct computation whenever at least one of $\tint f, \tint g, \tint h$ is zero we have
\begin{align}
  (n-1)\langle\hat{f}(\sigma)\hat{g}(\sigma),\hat{h}(\sigma)\rangle_\HS
  &= (n-1) \int_{S_n^2} (f*g)(x) \overline{h(y)} \tr\sigma(xy^{-1})\nonumber\\
  &= (n-1) \int_{S_n^2} (f*g)(x) \overline{h(y)} \(\sum_{i\in\Omega} 1_{x(i)=y(i)} - 1\)\nonumber\\
  &= (n-1) \sum_{i\in\Omega} \int_{S_n^2} (f*g)(x) \overline{h(y)} 1_{x(i)=y(i)}\nonumber\\
  &= \frac{n-1}n \sum_{i\in\Omega} \langle f*p_ig,p_ih\rangle\nonumber\\
  &\sim \sum_{i\in\Omega} \langle f*p_ig,p_ih\rangle\label{secondterm}.
\end{align}
Here we define the convolution of functions $f\in L^2(S_n)$ and $u\in L^2(\Omega)$ by the same formula:
\[
  f*u(\omega) = \int_{S_n} f(\pi) u(\pi^{-1}(\omega));
\]
$f*u$ is then a function defined on $\Omega$, and one may check the relation
\[
  p_i(f*g) = f*p_ig.
\]
Note that the assumption that one of $\tint f,\tint g, \tint h$ is zero is innocuous, since changing $f$ by a constant does not change $\hat{f}(\sigma)$.

Similarly whenever $\tint f = 0$ we have the following remnant of Parseval's identity:
\begin{equation}
\|f\|_2^2 \geq (n-1)\|\hat{f}(\sigma)\|^2_\HS \sim \sum_{i\in\Omega} \|p_i f\|_2^2.\label{remnantparseval}
\end{equation}

We can now summarize the rest of the proof. We will prove a concentration-of-measure result for the randomly rearranged inner product
\[
  \langle \pi*p_ig,p_ih\rangle = \int_\Omega p_ig(\pi^{-1}(\omega)) p_ih(\omega).
\]
This result will ensure that
\[
  \langle \pi*p_ig,p_ih\rangle \approx \tint p_ig \tint p_ih = \tint g \tint h
\]
with high probability, and with a tail depending on the variances $\|p_ig-\tint g\|_2^2$ and $\|p_ih - \tint h\|_2^2$ of $p_i g$ and $p_i h$ and on the entropies of $p_ig/\tint g$ and $p_ih/\tint h$. Crucially, when the variances are small there is rather strong concentration from below, unless one of the entropies is large. We will then apply the Parseval remnant~\eqref{remnantparseval} and a version of subadditivity of entropy to conclude.

\section{An inequality of Carlen, Lieb, and Loss}\label{sec:cll}

The following inequality was proved by Carlen, Lieb, and Loss~\cite{cll}.

\begin{theorem}\label{cll} Let $f_1,\dots,f_n:\Omega\to\C$ be functions. Then
\[
  \int_{S_n} \prod_{i=1}^n |f_i(\pi(i))| \leq \prod_{i=1}^n \|f_i\|_2.
\]
\end{theorem}

This inequality can be viewed in at least two ways. First, as it resembles the classical Loomis--Whitney inequality, or more generally the Brascamp--Lieb inequality, it can be viewed as an inequality of Brascamp--Lieb-type for the symmetric group. In this light Theorem~\ref{cll} bears a striking resemblance to another Brascamp--Lieb-type inequality proved by Carlen, Lieb, and Loss for the sphere: see~\cite{cll-sphere}.

Theorem~\ref{cll} can also be viewed as a Hadamard-type inequality for permanents. The classical Hadamard inequality states that if $M$ is a matrix with columns $v_1,\dots,v_n\in\C^n$ then
\[
  |\det(M)| \leq \prod_{i=1}^n |v_i|,
\]
where $|\cdot|$ is the usual Euclidean norm on $\C^n$. By comparison Theorem~\ref{cll} states that
\[
  |\textup{perm}(M)| \leq \frac{n!}{n^{n/2}} \prod_{i=1}^n |v_i|.
\]

In this section we deduce two consequences of Theorem~\ref{cll}, neither of them original: a version of entropy subadditivity for the symmetric group, and a concentration-of-measure result for a statistic of Hoeffding.

\subsection{Entropy subadditivity for the symmetric group}

Given $f:S_n\to[0,\infty)$ with $\alpha = \tint f$ we define the entropy of $f$ to be
\[
  S(f) = \int_{S_n} (f/\alpha) \log (f/\alpha).
\]
To be more precise we might call $S(f)$ the Kullback--Liebler divergence of $(f/\alpha)\,d\pi$ from uniform, but we will use the shorter term for simplicity. Similarly, given $g:\Omega\to[0,\infty)$ with $\beta = \tint g$ we define
\[
  S(g) = \int_\Omega (g/\beta) \log (g/\beta).
\]
All logarithms are of course taken to the natural base.

In the coup de gr\^ace of our argument we will apply the following entropy-subadditivity inequality.

\begin{theorem}[Subadditivity of entropy]\label{subadditivity}
Suppose $f:S_n\to [0,\infty)$. Then
\[
  S(f) \geq \frac12 \sum_{i\in\Omega} S(p_i f).
\]
\end{theorem}

Note that this is much stronger than what one gets from just applying usual entropy subadditivity to $f$ as a function $[n]^n\to[0,\infty)$. 

Theorems~\ref{cll} and~\ref{subadditivity} are more closely related than it may appear, as shown in some generality by Carlen and Cordero-Erausquin~\cite{cc}. We repeat the rather simple deduction of Theorem~\ref{subadditivity} from Theorem~\ref{cll} here for the convenience of the reader.

\begin{proof}[Proof of Theorem~\ref{subadditivity}]
Put $\alpha = \tint f$. Define $f':S_n\to[0,\infty)$ by
\[
  f'(\pi) = \prod_{i=1}^n p_if(\pi(i))^{1/2},
\]
and put $\alpha' = \tint f'$. Then by Jensen's inequality we have
\begin{align}
  0
  &\leq \int_{S_n} (f/\alpha) \log\(\frac{f/\alpha}{f'/\alpha'}\)\nonumber\\
  &= S(f) - \frac12 \sum_{i=1}^n \int_{S_n} (f/\alpha) \log p_if(\pi(i)) + \log \alpha'\nonumber\\
  &= S(f) - \frac12 \sum_{i=1}^n \int_{\Omega} (p_if/\alpha) \log p_if + \log\alpha'\nonumber\\
  &= S(f) - \frac12 \sum_{i=1}^n S(p_if) - \frac{n}{2} \log\alpha + \log \alpha'.\label{jensen}
\end{align}
On the other hand by Theorem~\ref{cll} we have
\[
  \alpha' = \int_{S_n} \prod_{i=1}^n p_if(\pi(i))^{1/2} \leq \prod_{i=1}^n \(\int_\Omega p_i f\)^{1/2} = \alpha^{n/2},
\]
so $\log \alpha' \leq \frac{n}{2} \log \alpha$ and the theorem follows from~\eqref{jensen}.
\end{proof}

\subsection{Concentration for Hoeffding's statistic}

Given an $n\times n$ complex matrix $(a_{ij})$ we consider the sum
\[
  X = \sum_{i=1}^n a_{i\pi(i)},
\]
where $\pi\in S_n$ is random permutation. The study of such sums goes back at least to Hoeffding~\cite{hoeffding}, who proved a central limit theorem for $X$ under suitable hypotheses, and so we refer to $X$ as \emph{Hoeffding's statistic}. More recently work on Hoeffding's statistic has been more or less wedded to Stein's method of exchangeable pairs, starting with Bolthausen's~\cite{bolthausen} Berry--Esseen-type estimate for the error in Hoeffding's theorem, and following with the work of Chatterjee~\cite{chatterjee}, who proved the first nonasymptotic concentration-type result for such sums.

In the next section we will need the following Bernstein-type concentration inequality for Hoeffding's statistic, which was proved in the more general context of random matrix theory by Mackey, Jordan, Chen, Farrell, and Tropp~\cite[Corollary~10.3]{mjcft}, using an extension of Chatterjee's method.

\begin{theorem}\label{bernstein}
Let $(a_{ij})$ be an $n\times n$ matrix such that $\sum_{i,j=1}^n a_{ij} = 0$ and such that $|a_{ij}|\leq M$ for each $i,j$. Let $v = \frac1n \sum_{i,j=1}^n |a_{ij}|^2$. Let $\pi\in S_n$ be chosen uniformly at random, and let
\[
  X = \sum_{i=1}^n a_{i\pi(i)}.
\]
Then for all $t>0$ we have
\[
  \P(|X|>t) \leq 2 \exp\( \frac{-ct^2}{v + Mt}\),
\]
where $c$ is some positive constant.
\end{theorem}

The purpose of this subsection is to give another proof of the above theorem, not relying on Stein's method, but instead relying on the Carlen--Lieb--Loss inequality Theorem~\ref{cll}. The main value of doing so is to reduce the reliance of the present paper on results proved elsewhere, but it may also be of independent interest.

\begin{proof}[Proof of Theorem~\ref{bernstein}]
By replacing $(a_{ij})$ with $(a_{ij} - \frac1n \sum_{j'} a_{ij'})$ if necessary and slightly reducing the constant $c$ we may assume that $\sum_j a_{ij} = 0$ for each $i$: note that this operation does not change $X$, it can at worst double $\max |a_{ij}|$, and it can only reduce $v$. We may also assume that $(a_{ij})$ is real, for otherwise we may just deal with the real and imaginary parts separately.

Now for $\lambda>0$ we have, by Theorem~\ref{cll},
\begin{align}
  \E\exp(\lambda X)
  &=\int_{S_n} \prod_{i=1}^n \exp(\lambda a_{i\pi(i)})\nonumber\\
  &\leq \prod_{i=1}^n \(\frac1n \sum_{j=1}^n \exp(2\lambda a_{ij})\)^{1/2}.\label{expmom}
\end{align}
Define
\[
  h(x) = \frac{e^x - 1 - x}{x^2} = \sum_{k=0}^\infty \frac{x^k}{(k+2)!}.
\]
Then
\begin{align*}
  \frac1n \sum_{j=1}^n \exp(2\lambda a_{ij})
  &= \frac1n \sum_{j=1}^n \( 1 + 2\lambda a_{ij} + 4\lambda^2 a_{ij}^2 h(2\lambda a_{ij})\)\\
  &= 1 + \frac1n \sum_{j=1}^n 4\lambda^2 a_{ij}^2 h(2\lambda a_{ij})\\
  &\leq 1 + 4\lambda^2 \(\frac1n \sum_{j=1}^n a_{ij}^2 \) h(2\lambda M)\\
  &\leq \exp\(4\lambda^2 \(\frac1n \sum_{j=1}^n a_{ij}^2\) h(2\lambda M)\),
\end{align*}
so from~\eqref{expmom} and the simple bound
\[
 h(x) \leq \sum_{k=0}^\infty x^k = \frac1{1-x} \qquad(0<x<1),
\]
we have
\[
  \E\exp(\lambda X) \leq \exp\(\frac{2\lambda^2 v}{1-2\lambda M}\)
\]
for $2\lambda M<1$. The claimed result now follows by bounding
\[
  \P(X>t) = \P(\exp(\lambda X) > e^{\lambda t}) \leq e^{-\lambda t} \E\exp(\lambda X) \leq \exp\(-\lambda t + \frac{2\lambda^2 v}{1 - 2\lambda M}\)
\]
and putting
\[
  \lambda = \frac{t}{4v + 2Mt},
\]
and similarly bounding $\P({-X}>t)$.
\end{proof}

The reader familiar with the usual Bernstein inequality may recognize that from~\eqref{expmom} onwards all we have done is reproduce the usual proof. Indeed, if $Y$ is the sum of $n$ independent random variables, the $i$th of which takes values $a_{i1},\dots, a_{in}$ each with probability $1/n$, then~\eqref{expmom} states that
\[
  \E \exp(\lambda X) \leq \(\E\exp(2\lambda Y)\)^{1/2},
\]
so it suffices to extract from the proof of the usual Bernstein inequality an upper bound for $\E\exp(2\lambda Y)$.

\section{Refined concentration for rearrangements}

In this section we prove a refined concentration estimate for Hoeffding's statistic
\[
  X = \sum_{i=1}^n a_{i\pi(i)}
\]
under the hypothesis that $a_{ij} = u_i v_j$ for some $(u_i)$ and $(v_i)$ for which we have some sort of entropy control. Moreover we are particularly interested in the concentration from below, which in certain regimes we expect to be stronger than the concentration from above.

\begin{theorem}\label{thm:rearrangement-concentration}
Let $f:S_n\to[0,1]$ be a function with $\tint f = \alpha$. Let $g_1,g_2:\Omega\to[0,1]$ be functions with $\tint g_1 = \beta$ and $\tint g_2 = \gamma$. Then
\begin{align*}
 - \(\langle f*g_1,g_2\rangle - \alpha \beta\gamma \)
 &\lesssim \frac{\alpha \|g_1 - \beta \|_2 \|g_2 - \gamma\|_2 \log n}{n^{1/2}}\\
 &\qquad+ \frac{\alpha^{1/2}\beta^{1/2}\gamma^{1/2}(\beta^{1/2}+\gamma^{1/2}) S(g_1)^{1/2} S(g_2)^{1/2} (\log n)^{5/2}}{n^{1/2}}\\
 &\qquad+ O(n^{-99}).
\end{align*}
\end{theorem}

\begin{lemma}\label{lem:rearrangement-concentration}
Let $h_1,h_2:\Omega\to[0,1]$ be functions such that $h_i$ is supported on a set $H_i$ of density $\delta_i$, and such that $1/2\leq h_i\leq 1$ on $H_i$. Let $f:S_n\to[0,1]$ be a function with $\tint f = \alpha$. Then if $\delta_1\delta_2\gtrsim n^{-1}$ we have
\[
  |\langle f*h_1,h_2\rangle - \alpha \tint h_1 \tint h_2| \lesssim \frac{\alpha \delta_1^{1/2} \delta_2^{1/2} \log n}{n^{1/2}} + O(n^{-100}),
\]
while if $\delta_1\delta_2 \lesssim n^{-1}$ we have
\[
  - \alpha\delta_1\delta_2 \lesssim \(\langle f\ast h_1,h_2\rangle - \alpha \tint h_1 \tint h_2\)  \lesssim \min\(\frac{\alpha\log n}{n} + O(n^{-100}),\delta_1\delta_2\).
\]
\end{lemma}

To explain the two cases appearing in Lemma~\ref{lem:rearrangement-concentration}, let us momentarily think of $h_i$ as the indicator of $H_i$. The inner product $\langle f*h_1,h_2\rangle/\alpha$ is then the density of a random intersection $\pi(H_1)\cap H_2$, where $\pi$ is chosen randomly according to $f/\alpha$. If $H_1$ and $H_2$ are not too small then we expect $|\pi(H_1)\cap H_2|$ to be highly concentrated around $\delta_1\delta_2 n$ with a Gaussian-type tail: this is the first case in the lemma. However if $H_1$ and $H_2$ are small then $|\pi(H_1)\cap H_2|$ has a Poisson-type distribution, so we expect $\pi(H_1)\cap H_2$ to be nonempty with probability about $\delta_1\delta_2$, and in any case almost surely bounded in size by about $\log n$: this is the second case in the lemma. The lower bound in the second case is trivial.

\begin{proof}
Apply Theorem~\ref{bernstein} to $a_{ij} = (h_1(i)-\tint h_1)(h_2(j) - \tint h_2)$, noting that $|a_{ij}|\leq 1$ and
\[
  \frac1n \sum_{i,j=1}^n a_{ij}^2 = \frac1n \sum_{i=1}^n (h_1(i)-\tint h_1)^2 \sum_{j=1}^n (h_2(j)-\tint h_2)^2 \lesssim \delta_1\delta_2 n.
\]
The result is that
\[
  \P(n|\langle \pi\ast(h_1-\tint h_1),(h_2-\tint h_2)\rangle| > t) \leq 2 \exp\( \frac{-ct^2}{\delta_1\delta_2n + t} \).
\]
Thus for every $t>0$ we have
\[
 | \langle f*(h_1-\tint h_1),(h_2-\tint h_2)\rangle| \lesssim \frac{\alpha t}{n} + 2 \exp\( \frac{-ct^2}{\delta_1\delta_2n + t} \).
\]

For the first part of the lemma put $t = C\delta_1^{1/2}\delta_2^{1/2}n^{1/2}\log n$ for some constant $C$. Then we obtain
\[
  | \langle f*(h_1-\tint h_1),(h_2-\tint h_2)\rangle| \lesssim_C \frac{\alpha \delta_1^{1/2} \delta_2^{1/2} \log n}{n^{1/2}} + 2 \exp\( \frac{-c C^2 (\log n)^2}{1 + (\delta_1\delta_2n)^{-1/2} C \log n} \).
\]
If $\delta_1\delta_2 \gtrsim n^{-1}$ and $C$ is sufficiently large it follows that
\[
  | \langle f*(h_1-\tint h_1),(h_2-\tint h_2)\rangle| \lesssim \frac{\alpha\delta_1^{1/2}\delta_2^{1/2} \log n}{n^{1/2}} + O(n^{-100}),
\]
as claimed.

For the second part of the lemma put $t = C \log n$ for some constant $C$. Then we obtain
\[
  | \langle f*(h_1-\tint h_1),(h_2-\tint h_2)\rangle| \lesssim_C \frac{\alpha \log n}{n} + 2 \exp\( \frac{-c C^2 (\log n)^2}{\delta_1\delta_2 n + C \log n} \).
\]
Now if $\delta_1\delta_2\lesssim n^{-1}$ and $C$ is sufficiently large it follows that
\[
  | \langle f*(h_1-\tint h_1),(h_2-\tint h_2)\rangle| \lesssim \frac{\alpha \log n}{n} + O(n^{-100}).
\]

The remaining inequalities asserted by the lemma are trivial: just note that
\[
  \langle f*h_1,h_2\rangle \leq \langle 1*h_1,h_2\rangle = \tint h_1\tint h_2 \lesssim \delta_1\delta_2,
\]
and
\[
  \alpha\tint h_1\tint h_2 \lesssim \alpha\delta_1\delta_2.\qedhere
\]
\end{proof}

We will deduce Theorem~\ref{thm:rearrangement-concentration} from Lemma~\ref{lem:rearrangement-concentration} using a dyadic decomposition, but first we need two basic entropy computations.

\begin{lemma}\label{lem:entropy-low}
Let $g:\Omega\to[0,1]$ be a function such that $\tint g = \beta$ and such that $g\leq \beta - t$ on a set of density at least $\delta$, where $t,\delta>0$. Then
\[
  S(g) \gtrsim \frac{\delta t^2}{\beta^2}.
\]
\end{lemma}
\begin{proof}
We must have $t\leq \beta$, so by replacing $t$ with $t/100$ if necessary we may assume that $t/\beta \leq 1/100$. Similarly, by reducing $\delta$ if necessary we may assume that $\delta\leq 1/2$ and that $\delta n$ is an integer. Now by convexity $S(g)$ is minimized under the stated conditions when $g = \beta-t$ on a set of density $\delta$ and otherwise equal to $\beta + \frac{\delta}{1-\delta} t$, and in this case
\[
  S(g) = \delta \(1-\frac{t}{\beta}\) \log\(1 - \frac{t}{\beta}\) + (1-\delta) \(1 + \frac{\delta}{1-\delta} \frac{t}{\beta}\) \log\(1 + \frac{\delta}{1-\delta}\frac{t}{\beta}\).
\]
By inserting the Taylor expansion
\begin{equation}\label{taylor}
  (1+x)\log(1+x) = x + x^2/2 + O(x^3)
\end{equation}
we thus have
\[
  S(g) \geq \frac12 \frac{\delta}{1-\delta} \frac{t^2}{\beta^2} + O\(\delta \frac{t^3}{\beta^3}\) \gtrsim \frac{\delta t^2}{\beta^2}.
\]
The last inequality follows from our assumption $t/\beta\leq 1/100$.
\end{proof}

\begin{lemma}\label{lem:entropy-high}
Let $g:\Omega\to[0,1]$ be a function such that $\tint g = \beta$ and such that $g \geq \beta + t$ on a set of density at least $\delta$, where $t,\delta>0$. Then
\[
  S(g) \gtrsim \min\(\frac{\delta t}{\beta}, \frac{\delta t^2}{\beta^2}\) \geq \frac{\delta t^2}{\beta}.
\]
\end{lemma}
\begin{proof}
We must have $(\beta+t)\delta \leq \tint g = \beta$, i.e.,
\[
  \frac{\delta}{1-\delta} \frac{t}{\beta}\leq 1,
\]
so by replacing $t$ with $t/100$ if necessary we may assume that
\[
  \frac{\delta}{1-\delta} \frac{t}{\beta}\leq \frac1{100}.
\]
As before we may also assume that $\delta\leq 1/2$ and that $\delta n$ is an integer. Now by convexity $S(g)$ is minimized under the stated conditions when $g = \beta+t$ on a set of density $\delta$ and otherwise equal to $\beta - \frac{\delta}{1-\delta} t$, and in this case
\[
  S(g) = \delta \(1+\frac{t}{\beta}\) \log\(1 + \frac{t}{\beta}\) + (1-\delta) \(1 - \frac{\delta}{1-\delta} \frac{t}{\beta}\) \log\(1 - \frac{\delta}{1-\delta}\frac{t}{\beta}\).
\]
By inserting~\eqref{taylor} we thus have
\[
  S(g) \geq \delta \(1+\frac{t}{\beta}\) \log\(1 + \frac{t}{\beta}\) - \delta \frac{t}{\beta} + \frac12 \frac{\delta^2}{(1-\delta)} \frac{t^2}{\beta^2} + O\(\frac{\delta^3t^3}{\beta^3}\).
\]

Now we separate into cases depending on the size of $t/\beta$. If $t/\beta \geq 1$ then we have
\[
  S(g) \geq \delta \frac{t}{\beta} (2\log 2-1) + O\(\frac{\delta^2t^2}{\beta^2}\) \gtrsim \frac{\delta t}{\beta}.
\]
On the other hand if $t/\beta \leq 1$ then by reducing $t$ if necessary we may assume that $t/\beta \leq 1/100$, and then by inserting~\eqref{taylor} again we have
\[
  S(g) \geq \frac12 \frac{\delta}{1-\delta} \frac{t^2}{\beta^2} + O\(\frac{\delta t^3}{\beta^3}\) \gtrsim \frac{\delta t^2}{\beta^2}.
\]
As before we used our assumption about the size of $\delta t/\beta$ or $t/\beta$ to justify the absorption of the error terms.
\end{proof}

\begin{proof}[Proof of Theorem~\ref{thm:rearrangement-concentration}]
Write
\[
  g_i - \tint g_i = \sum_s g_i^s + O(n^{-100}) = \sum_s (g_i^s - \tint g_i^s) + O(n^{-100}),
\]
where $s$ ranges over all $s$ of the form $\pm 2^{-k}$ for which $n^{-100}\leq |s|\leq 1$, and where $g_i^s$ is defined to be equal to $g_i-\tint g_i$ where $g_i-\tint g_i$ has the same sign as $s$ and $|s|/2 < |g_i - \tint g_i| \leq |s|$ and zero elsewhere. Then 
\[
  \langle f*g_1, g_2\rangle - \alpha\beta\gamma = \sum_{s,t} \(\langle f*g_1^s,g_2^t\rangle - \alpha \tint g_1^s \tint g_2^t \) + O(n^{-100}).
\]

For each $s,t$ we apply Lemma~\ref{lem:rearrangement-concentration} with $h_1 = g_1^s/s$ and $h_t=g_2^t/t$. Let $\delta_1^s$ be the density of points where $g_1-\tint g_1$ has the same sign as $s$ and $|s|/2 < |g_1 - \tint g_1| \leq |s|$ and let $\delta_2^t$ be the density of points where $g_2-\tint g_2$ has the same sign as $t$ and $|t|/2 < |g_2-\tint g_2|\leq |t|$. If $\delta_1^s \delta_2^t \gtrsim 1/n$ then we get the bound
\[
 \left|\langle f*g_1^s, g_2^t\rangle - \alpha \tint g_1^s \tint g_2^t\right| \lesssim \frac{\alpha |s| |t| (\delta_1^s)^{1/2} (\delta_2^t)^{1/2} \log n}{n^{1/2}} + O(n^{-100}),
\]
and the total contribution from all such cases is bounded by
\begin{align*}
  &\sum_{s,t} \(\frac{ \alpha |s| |t| (\delta_1^s)^{1/2} (\delta_2^t)^{1/2} \log n}{n^{1/2}} + O(n^{-100})\)\\
  &\qquad\lesssim \frac{\alpha \|g_1-\tint g_1\|_2 \|g_2-\tint g_2\|_2 \log n}{n^{1/2}} + O(n^{-99}).
\end{align*}

Now consider the cases in which $\delta_1^s\delta_2^t \lesssim 1/n$ and in which $s$ and $t$ have the same sign. By Lemma~\ref{lem:rearrangement-concentration} we have
\[
  -\(\langle f\ast g_1^s,g_2^t\rangle - \alpha g_1^s g_2^t\) \lesssim \alpha |s| |t| \delta_1^s \delta_2^t \lesssim \frac{\alpha |s| |t| (\delta_1^s)^{1/2} (\delta_2^t)^{1/2}}{n^{1/2}},
\]
so the total contribution from these cases is again acceptable.

Finally consider the cases in which $\delta_1^s\delta_2^t\lesssim 1/n$ and in which $s$ and $t$ have opposite sign, say $s<0$ and $t>0$. By Lemmas~\ref{lem:entropy-low} and~\ref{lem:entropy-high} we have
\[
  S(g_1) \gtrsim \frac{\delta_1^s s^2}{\beta^2}
\]
and
\[
  S(g_2) \gtrsim \frac{\delta_2^t t^2}{\gamma},
\]
so
\[
  S(g_1)^{1/2} S(g_2)^{1/2} \gtrsim \frac{|s| |t| (\delta_1^s \delta_2^t)^{1/2}}{\beta\gamma^{1/2}}.
\]
Thus by Lemma~\ref{lem:rearrangement-concentration} we can bound
\begin{align*}
  |\langle f*g_1^s,g_2^t\rangle - \alpha \tint g_1^s \tint g_2^t |
  &\lesssim |s| |t| \(\frac{\alpha \log n}{n}\)^{1/2} \(\delta_1^s \delta_2^t\)^{1/2} + O(n^{-100})\\
  &\lesssim \frac{\alpha^{1/2} \beta \gamma^{1/2} S(g_1)^{1/2} S(g_2)^{1/2} (\log n)^{1/2}}{n^{1/2}} + O(n^{-100}).
\end{align*}
If $s<0$ and $t>0$ then we get the analogous bound
\[
  |\langle f*g_1^s,g_2^t\rangle - \alpha \tint g_1^s \tint g_2^t |
  \lesssim \frac{\alpha^{1/2} \beta^{1/2} \gamma S(g_1)^{1/2} S(g_2)^{1/2} (\log n)^{1/2}}{n^{1/2}} + O(n^{-100}).
\]
The number of choices of $s$ and $t$ is bounded by $(\log n)^2$, so the total contribution from all these cases is bounded by
\[
  \frac{\alpha^{1/2}\beta^{1/2}\gamma^{1/2}(\beta^{1/2} + \gamma^{1/2})S(g_1)^{1/2} S(g_2)^{1/2} (\log n)^{5/2}}{n^{1/2}} + O(n^{-99}).\qedhere
\]
\end{proof}

\section{Bounding the second term in~\eqref{ellisgreen}}

\begin{proof}[Proof of Theorem~\ref{main}]
Let $f=1_X$, $g=1_Y$, and $h=1_Z$, where $X,Y,Z\subset A_n$ have densities $\alpha,\beta,\gamma \geq n^{-O(1)}$ respectively. Then the first term in~\eqref{ellisgreen} is
\[
	\alpha\beta\gamma,
\]
the third term is bounded by
\[
	c\alpha^{1/2}\beta^{1/2}\gamma^{1/2}/n,
\]
and the second term is, by~\eqref{secondterm} and Theorem~\ref{thm:rearrangement-concentration},
\begin{align*}
	(n-1)&\langle\hat{f}(\sigma)\hat{g}(\sigma),\hat{h}(\sigma)\rangle_\HS\\
	&\sim \sum_{i\in\Omega} \langle f*(p_ig-\beta),(p_ih-\gamma)\rangle\\
	&\gtrsim - \frac{\alpha \log n}{n^{1/2}} \sum_{i\in\Omega} \|p_i g - \beta \|_2 \|p_i h - \gamma\|_2 \\
	&\qquad - \frac{\alpha^{1/2}\beta^{1/2}\gamma^{1/2}(\beta^{1/2}+\gamma^{1/2})(\log n)^{5/2}}{n^{1/2}} \sum_{i\in\Omega} S(p_i g)^{1/2} S(p_i h)^{1/2}\\
	&\qquad + O(n^{-98}).
\end{align*}
By Cauchy--Schwarz and the Parseval remnant~\eqref{remnantparseval}, the first term here is bounded in magnitude by
\[
	\frac{\alpha\log n}{n^{1/2}} \(\sum_{i\in\Omega} \|p_ig - \beta\|_2^2\)^{1/2} \(\sum_{i\in\Omega} \|p_ih - \gamma\|_2^2\)^{1/2} \lesssim \frac{\alpha \beta^{1/2} \gamma^{1/2} \log n}{n^{1/2}}.
\]
Similarly, by Cauchy--Schwarz and subadditivity of entropy (Theorem~\ref{subadditivity}) the second term is bounded in magnitude by
\begin{align*}
  &\frac{\alpha^{1/2}\beta^{1/2}\gamma^{1/2}(\beta^{1/2}+\gamma^{1/2})(\log n)^{5/2}}{n^{1/2}} \(\sum_{i\in\Omega} S(p_i g)\)^{1/2} \(\sum_{i\in\Omega} S(p_i h)\)^{1/2}\\
  &\qquad \lesssim \frac{\alpha^{1/2}\beta^{1/2}\gamma^{1/2}(\beta^{1/2}+\gamma^{1/2})(\log n)^{5/2}}{n^{1/2}} (\log\beta^{-1})^{1/2} (\log\gamma^{-1})^{1/2}\\
  &\qquad \lesssim \frac{\alpha^{1/2}\beta^{1/2}\gamma^{1/2}(\beta^{1/2}+\gamma^{1/2})(\log n)^{7/2}}{n^{1/2}}.
\end{align*}
Thus we deduce that $\langle f*g,h\rangle \geq (1+o(1))\alpha\beta\gamma$ provided that
\begin{align*}
  \frac{\alpha^{1/2}\beta^{1/2}\gamma^{1/2}}{n} &\ll \alpha\beta\gamma,\\
  \frac{\alpha\beta^{1/2}\gamma^{1/2}(\log n)}{n^{1/2}} &\ll \alpha\beta\gamma,\\
  \frac{\alpha^{1/2}\beta\gamma^{1/2}(\log n)^{7/2}}{n^{1/2}} &\ll \alpha\beta\gamma,\\
  \frac{\alpha^{1/2}\beta^{1/2}\gamma(\log n)^{7/2}}{n^{1/2}} &\ll \alpha\beta\gamma,~\text{and}\\
  n^{-98} &\ll \alpha\beta\gamma.
\end{align*}
In other words what we require is that
\[
  \min(\alpha\beta,\alpha\gamma,\beta\gamma) \gg (\log n)^7/n.\qedhere
\]
\end{proof}

\section{Open questions}

The most obvious outstanding open question is whether the logarithms can be removed from Theorem~\ref{main}. Specifically, does the largest product-free subset of $A_n$ have density $O(n^{-1/2})$? Can you say anything about the extremal examples? It is possible that all near-extremizers look roughly like the first example in Section~\ref{examples}, or its inverse, but this may be difficult to quantify, and even more difficult to prove.

Another obvious outstanding open question is whether a one-sided product-mixing phenomenon persists in other groups for densities lower than that given by Theorem~\ref{gowers}. For example take $G=\SL_2(p)$. For this group $m\sim p$. By Theorem~\ref{gowers} there is two-sided product mixing for sets of density at least $p^{-1/3}$, by Proposition~\ref{prop:many} there is no two-sided product mixing for sets of density less than $p^{-1/3}$, and by Proposition~\ref{prop:ked} there is no product mixing at all below density $p^{-1/2}$. Do we have one-sided product mixing for sets of densities between $p^{-1/2}$ and $p^{-1/3}$?

Another great question, which has been asked before by both Kedlaya~\cite{kedlayaAMM} and Gowers~\cite{gowers}, is about the product-mixing properties of $\textup{SU}(n)$. To make the question concrete, what is the measure of the largest product-free subset of $\textup{SU}(n)$? By straightforward adaptation of Theorem~\ref{gowers} it is at most $O(n^{-1/3})$, but the only lower bounds we know have the form $c^n$ for some $c<1$. Apart from being an interesting and natural question in its own right, answering this question may be relevant for understanding the product-mixing behaviour of groups not having a permutation representation of dimension $\sim m$.

\bibliography{productfreeAn}

\newcommand{\etalchar}[1]{$^{#1}$}
\begin{thebibliography}{MJC{\etalchar{+}}14}

\bibitem[BNP08]{babainikolovpyber}
L.~Babai, N.~Nikolov, and L.~Pyber.
\newblock Product growth and mixing in finite groups.
\newblock In {\em Proceedings of the {N}ineteenth {A}nnual {ACM}-{SIAM}
  {S}ymposium on {D}iscrete {A}lgorithms}, pages 248--257. ACM, New York, 2008.

\bibitem[Bol84]{bolthausen}
E.~Bolthausen.
\newblock An estimate of the remainder in a combinatorial central limit
  theorem.
\newblock {\em Z. Wahrsch. Verw. Gebiete}, 66(3):379--386, 1984.

\bibitem[CCE09]{cc}
E.~A. Carlen and D.~Cordero-Erausquin.
\newblock Subadditivity of the entropy and its relation to {B}rascamp-{L}ieb
  type inequalities.
\newblock {\em Geom. Funct. Anal.}, 19(2):373--405, 2009.

\bibitem[Cha07]{chatterjee}
S.~Chatterjee.
\newblock Stein's method for concentration inequalities.
\newblock {\em Probab. Theory Related Fields}, 138(1-2):305--321, 2007.

\bibitem[CLL04]{cll-sphere}
E.~A. Carlen, E.~H. Lieb, and M.~Loss.
\newblock A sharp analog of {Y}oung's inequality on {$S\sp N$} and related
  entropy inequalities.
\newblock {\em J. Geom. Anal.}, 14(3):487--520, 2004.

\bibitem[CLL06]{cll}
E.~A. Carlen, E.~H. Lieb, and M.~Loss.
\newblock An inequality of {H}adamard type for permanents.
\newblock {\em Methods Appl. Anal.}, 13(1):1--17, 2006.

\bibitem[Gow08]{gowers}
W.~T. Gowers.
\newblock Quasirandom groups.
\newblock {\em Combin. Probab. Comput.}, 17(3):363--387, 2008.

\bibitem[Hoe51]{hoeffding}
W.~Hoeffding.
\newblock A combinatorial central limit theorem.
\newblock {\em Ann. Math. Statistics}, 22:558--566, 1951.

\bibitem[Ked97]{kedlaya1}
K.~S. Kedlaya.
\newblock Large product-free subsets of finite groups.
\newblock {\em J. Combin. Theory Ser. A}, 77(2):339--343, 1997.

\bibitem[Ked98]{kedlayaAMM}
K.~S. Kedlaya.
\newblock Product-free subsets of groups.
\newblock {\em Amer. Math. Monthly}, 105(10):900--906, 1998.

\bibitem[MJC{\etalchar{+}}14]{mjcft}
L.~Mackey, M.~I. Jordan, R.~Y. Chen, B.~Farrell, and J.~A. Tropp.
\newblock Matrix concentration inequalities via the method of exchangeable
  pairs.
\newblock {\em Ann. Probab.}, 42(3):906--945, 2014.

\bibitem[Ras77]{rasala}
R.~Rasala.
\newblock On the minimal degrees of characters of {$S\sb{n}$}.
\newblock {\em J. Algebra}, 45(1):132--181, 1977.

\bibitem[Tao14]{tao}
T.~Tao.
\newblock {\em Hilbert's fifth problem and related topics}, volume 153 of {\em
  Graduate Studies in Mathematics}.
\newblock American Mathematical Society, Providence, RI, 2014.

\bibitem[Wig10]{wigderson}
A.~Wigderson.
\newblock Lecture notes for the 22nd {McGill} invitational workshop on
  computational complexity, 2010.
\newblock \url{http://www.math.ias.edu/~avi/TALKS/additive-lectures-v2.pdf}.

\end{thebibliography}
\bibliographystyle{alpha}

\begin{dajauthors}
\begin{authorinfo}[eberhard]
  Sean Eberhard\\
  London, UK\\
  eberhard\imagedot{}math\imageat{}gmail\imagedot{}com
\end{authorinfo}
\end{dajauthors}

\end{document}